\documentclass[12pt]{article} 
\usepackage{latexsym,amsmath}
\usepackage{graphicx}
\def\separation{\medskip}

\def\P{I\!\!P}

\def\C{{\cal C}}

\def\?{{\bf ??}}

\def\dim{{\rm dim}}

\def\ker{{\rm ker}}

 \newtheorem{theorem}{Theorem}[section]
\newtheorem{lemma}[theorem]{Lemma} 
\newtheorem{prop}[theorem]{Proposition} 
\newtheorem{definition}[theorem]{Definition} 
\newtheorem{corollary}[theorem]{Corollary}
 
\newtheorem{remark}[theorem]{Remark}

\newcommand{\proof}{{\it Proof.}\ } 
\newcommand{\qed}{\hfill  $\Box$\separation}

\begin{document}

\title{A computational approach to L\"uroth quartics}
 \author{Giorgio Ottaviani\footnote{The author is member of GNSAGA-INDAM.}} 
\date{} 
 \maketitle
 \begin{abstract}A plane quartic curve is called L\"uroth
if it contains the ten vertices of a complete pentalateral.
White and Miller constructed in 1909 a covariant quartic 4-fold,
associated to any plane quartic.
We review their construction and we show how it gives a computational tool to detect if a plane quartic
is L\"uroth. As a byproduct,  the 28 bitangents of a general plane quartic correspond to  28
singular points of the associated White-Miller quartic 4-fold.
\end{abstract}
\section*{Introduction}
A {\it L\"uroth quartic} is a 
 plane quartic containing the ten vertices of a complete 
{\it pentalateral},
like in Figure 1.
\begin{figure}
    \centerline{\includegraphics[width=60mm]{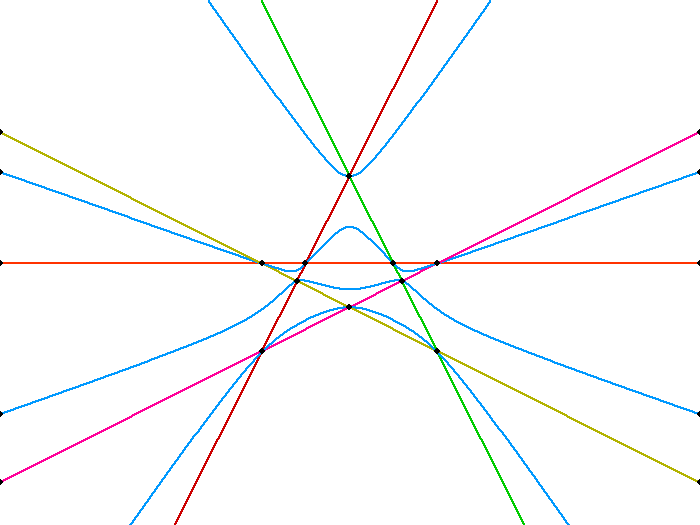}}
    \caption{A L\"uroth quartic with a pentalateral inscribed, plotted with XAlci.}
    \end{figure}
 The pentalateral is inscribed in the
quartic, and (equivalently) the quartic  circumscribes the pentalateral.
The L\"uroth quartics attracted  a lot of attention
in the classical literature, because they show  a ``poristic''
phenomenon, namely ``the fallacy of constant counting'',
to use the words of White and Miller, \cite{WM}. 

Indeed, L\"uroth was the first to observe, in 1868,  that  when a plane quartic circumscribes one pentalateral,
it also circumscribes infinitely many other pentalaterals, which is not expected by a naive dimension
count (see \cite{WM} pag. 348). 
As a consequence of L\"uroth's result, L\"uroth quartics are not dense in the space of quartics,
but they fill an open subset of a hypersurface of $\P^{14}$.
The equation of this hypersurface is called the {\it L\"uroth invariant} and,
to the best of our knowledge, it is still unknown.
This hypersurface consists entirely of quartics, so the limits (when the pentalateral degenerates) are also called nowadays L\"uroth quartics.
L\"uroth found a one dimensional
family of inscribed pentalaterals,
{\it all tangent to the same conic}. There are examples
with more  families of pentalaterals (see \S 4), and so more conics.
Each family defines a particular even theta characteristic, called {\it pentalateral theta}.
Let $\delta$ be the number of pentalateral theta on a general
L\"uroth quartic. It is a well-established fact, already known to classical geometers, that the L\"uroth invariant
has degree $54/\delta$ (see the proof of Theorem \ref{morley}).
The classical study on L\"uroth quartics culminated in 1919 in Morley's proof\cite{fM14}
that the L\"uroth invariant has degree $54$. We reviewed this lovely proof in a joint paper with E. Sernesi\cite{OS1}, and we refer to 
\cite{OS1, D, PT, Tyu} for more information, including the role
of L\"uroth quartics in vector bundle theory.
In equivalent way, Morley proved that $\delta=1$, that is that 
for the general L\"uroth quartic there is a unique
family of inscribed pentalaterals,
all tangent to the same conic, which is
 uniquely determined. We emphasized this point of view, even if it was not
focal in \cite{fM14}, because it is the starting  point of  our
current paper.

We try to answer to the following 

{\bf Question:}
given an explicit quartic curve given by a homogeneous quartic polynomial in three variables, how can one detect
if it is L\"uroth?

The answer should follow from the explicit expression of the L\"uroth invariant.
Since its expression is unclear, we have to look for alternative roads. 

For this purpose, we review a construction due to White and Miller\cite{WM}.
They construct, for any plane quartic curve $f$, a covariant quartic 4-fold
 $\mathrm{WM}_f$ in the space $\P^5$ of conics.
They proved that when $f$ is L\"uroth with a pentalateral theta corresponding to the conic $q$,
the point $q\in\P^5$ is a singular point of $\mathrm{WM}_f$. White-Miller's aim was to use
this property to characterize the L\"uroth property, with the hope that the discriminant
of $\mathrm{WM}_f$ could be a power of the L\"uroth invariant.
Unfortunately, this attempt fails, because the discriminant
of $\mathrm{WM}_f$ vanishes identically.

Indeed we prove 

\begin{theorem}\label{28bit}
For the general plane quartic $f$, $\mathrm{WM}_f$ has always $28$ singular points
corresponding to the $28$ bitangents (odd theta characteristics) of $f$.
\end{theorem}

This theorem owes a lot to the computer experiments performed with Macaulay2\cite{GS}.
The failure of White-Miller's hope  is balanced by two
remarks, which give,in some sense, a reprise to that hope, and show that
the quartic $\mathrm{WM}_f$ is an interesting object. The first one is that Theorem \ref{28bit} gives a way to compute
the ideal of the bitangents, and
it gives at once the information about the bitangent lines and the tangency points.

The second remark is that it allows one to partially answer our question of how to detect  if a quartic is L\"uroth.
Indeed, with Macaulay2, it is possible to quickly compute the singular locus of $\mathrm{WM}_f$,
given any $f\in S^4V$.
The following outcomes are possible

$\bullet$(i) if the singular locus of $\mathrm{WM}_f$ has dimension zero and degree $28$, then $f$ {\it is not L\"uroth}.
This outcome is the general one.

$\bullet$(ii) if the singular locus of $\mathrm{WM}_f$ has dimension zero and degree $29$, the ``extra'' point can be found
by quotient over the ideal generated by the cubic hypersurface of singular conics. When it remains a single point,
it corresponds to a smooth conic. In this case, after the check
of a mild open condition (see \S \ref{algorithm}), $f$ {\it is L\"uroth and the conic corresponds to the unique pentalateral theta}.
This outcome is found for the general L\"uroth curve $f$. 
In this case the explicit equation of  all the pentalaterals inscribed can be found. 

$\bullet$(iii) if the outcome is different from (i) and (ii) a further analysis is necessary. 
The double irreducible conics, which are not L\"uroth (see the proof of Prop. 4.1 in \cite{OS2}), lie in this third category. If $f$ is a double irreducible conic,
then the singular locus of $\mathrm{WM}_f$ consists of
a surface of degree $10$, whose general point is a singular conic, plus a curve of degree $4$,
whose general point corresponds to a smooth conic .
Also the desmic quartics, which are L\"uroth and have at least
$6$ families of inscribed pentalateral (see \cite{hB14} pag. 367), lie in this third category.
When $f$ is desmic, $\mathrm{WM}_f$ is non reduced and it is a double hyperquadric in $\P^5$
(see Section \ref{appendix}).

As a consequence we get a {\it computer-aided proof  that the L\"uroth invariant has degree $54$}.
This proof is conceptually simpler than the two known proofs,
respectively by Morley\cite{fM14,OS1} and by LePotier-Tikhomirov\cite{PT}, but it cannot be concluded  without the help of a computer.

The material collected here owes a lot to the many thorough discussions
I had on the topic with Edoardo Sernesi during the preparation
of \cite{OS1} and \cite{OS2}. It is a pleasure to thank Edoardo for his insight and
 method. I thank also Igor Dolgachev and Bernd Sturmfels for their interest and many useful comments.

\section{Apolarity and the cubic invariant
of plane quartics}

Given any complex vector space $U$, we denote by $U^{\vee}$ its dual space.
Let $S^dU$ be the $d$-th symmetric power of $U$, its elements are 
homogeneous polynomials of degree $d$ in any coordinate system for $U$.
Any $f\in S^dU$ may be identified (up to scalar multiples) with its zero scheme in the projective space
$\P(U)$ of hyperplanes in $U$, so $f$ denotes also a hypersurface of degree $d$.

We recall a few facts about apolarity \cite{DK,RS}.
A polynomial $f\in S^dU$ is called {\it apolar} to a polynomial
$g\in S^eU^{\vee}$ if the contraction $f\cdot g\in S^{d-e}U$ is zero.
It is convenient to consider $g$ as a differential operator acting over $f$.
In the case $d=e$, the symmetric convention, that $f$ acts over $g$, works as well.

When $\dim U=2$, that is for polynomials over a projective line,
the apolarity is well defined for $f$, $g$ both in $S^dU$.
This is due to the canonical isomorphism $U\simeq U^{\vee}\otimes\wedge^2U$.
Let $(x_0,x_1)$ be coordinates on $U$.
If $f=(a_0x_0+a_1x_1)^d$ and $g=(b_0x_0+b_1x_1)^d$,
then the contraction between $f$ and $g$ is easily seen to be
proportional to $(a_0b_1-a_1b_0)^d$.
This computation extends by linearity to any pair $f, g\in S^dU$,
because any polynomial can be expressed as a sum of $d$-th powers.
The resulting formula for $f=\sum_{i=0}^d{d\choose i}f_ix^{d-i}y^i$ and
$g=\sum_{i=0}^d{d\choose i}g_ix^{d-i}y^i$ is that $f$ is apolar to $g$ if and only if
$$\sum_{i=0}^d(-1)^i{d\choose i}f_ig_{d-i}=0$$

In particular 
\begin{lemma}\label{divides}Let $p, l^d\in S^dU$.
$p$ is apolar to $l^d$ if and only if $l$ divides $p$.
\end{lemma}

A polynomial $f\in S^4U$ is called equianharmonic if its apolar to itself.
So $f$ is equianharmonic if and only if
$$f_0f_4-4f_1f_3+3f_2^2=0$$
which is the expression for the classical invariant $I$ of binary quartics.

Let $(x_0,x_1,x_2)$ be coordinates on a $3$-dimensional complex space $V$ and
$(y_0,y_1,y_2)$ be coordinates on  $V^{\vee}$. 
Let $$f(x_0,x_1,x_2)=\sum_{i+j+k=4}\frac{4!}{i!j!k!}f_{ijk}x_0^ix_1^jx_2^k\in S^4V$$
be the equation of a plane quartic curve on $\P(V)$.
All invariants of $f$ have degree which is multiple of $3$ (\cite{Do2,Stu}).
The invariant of smallest degree has degree $3$ and it corresponds to a trilinear form $A(f,g,h)$, for $f, g, h\in S^4V$.
It is defined as
{\scriptsize
$$A\left((a_0x_0+a_1x_1+a_2x_2)^4,(b_0x_0+b_1x_1+b_2x_2)^4,(c_0x_0+c_1x_1+c_2x_2)^4\right)=
\left|\begin{array}{ccc}
a_0&a_1&a_2\\
b_0&b_1&b_2\\
c_0&c_1&c_2\\
\end{array}\right|^4$$
}
This definition extends by linearity to any $f, g, h\in S^4V$.
The explicit expression of the cubic invariant $A$ can be found
at art. 293 of Salmon's book\cite{Sal}, it can be checked today e.g. with Macaulay2 (\cite{GS}) and it is the sum of the following $23$ terms
(we denote $A(f)$ for $A(f,f,f)$)
{\scriptsize
\begin{gather}
\label{cubicinvariant}
A(f)=f_{400}f_{040}f_{004}+3(f_{220}^2f_{004}+f_{202}^2f_{040}+f_{400}f_{022}^2)
+12(f_{202}f_{121}^2+f_{220}f_{112}^2+f_{022}f_{211}^2)+6f_{220}f_{202}f_{022}+\nonumber\\
-4(f_{301}f_{103}f_{040}+f_{400}f_{031}f_{013}+f_{310}f_{130}f_{004})
+4(f_{310}f_{103}f_{031}+f_{301}f_{130}f_{013})+\nonumber\\
-12(f_{202}f_{130}f_{112}+f_{220}f_{121}f_{103}+f_{211}f_{202}f_{031}+f_{301}f_{121}f_{022}+
f_{310}f_{112}f_{022}+f_{220}f_{211}f_{013}+f_{211}f_{121}f_{112})+\nonumber\\
+12(f_{310}f_{121}f_{013}+f_{211}f_{130}f_{103}+f_{301}f_{112}f_{031})
\end{gather}
}

We will see in Remark \ref{altconstruction} a more geometric way to recover the same expression.

We will need in the sequel another expression for the
cubic invariant, borrowed from \cite{WM}. Call $\overline{f}=(f_{400},\ldots,f_{004})^t$ the (column)vector of coefficients of $f$.
The trilinear form can be written in the form

\begin{equation}\label{lg}A(f,g,h)=\overline{f}^tL_{g}\overline{h}\end{equation}

where $L_g$ is a $15\times 15$ symmetric matrix with entries linear in $g$.
The matrix $L_g$ encodes all the information to express the cubic invariant. It is denoted $L$ at page 349 of \cite{WM}, 
its explicit expression  
is reported in the appendix .

Note that given $f, g\in S^4V$, the equation $A(f,g,*)=0$  defines an element in the dual space $S^4V^{\vee}$,
possibly vanishing.

\begin{prop}\label{alternativedef}
(i) $A(f,g,l^4)=0$ if and only if the restrictions $f_{|l}$, $g_{|l}$
to the line $l=0$ are apolar.

(ii) Let $A(f,g,*)=H$. We have  $A(f,g,l^4)=0$ if and only if $H(l)=0$ .
\end{prop}
\begin{proof} To prove (i) , consider 
$f=(\sum_{i=0}^2a_ix_i)^4$, $g=(\sum_{i=0}^2b_ix_i)^4$, $l=x_2$.
Then $$A\left(f,g,l^4\right)=
\left|\begin{array}{ccc}
a_0&a_1&a_2\\
b_0&b_1&b_2\\
0&0&1\\
\end{array}\right|^4=
\left|\begin{array}{cc}
a_0&a_1\\
b_0&b_1\\
\end{array}\right|^4=f_{|l}\cdot g_{|l}$$
This formula extends by linearity to any $f$, $g$.

(ii) follows because $H(l)=H\cdot l^4$.
\end{proof}

\begin{remark}\label{altconstruction}The part (i) of Proposition \ref{alternativedef} gives an 
alternative way to define $L_g$.
Let $l$ with equation $\sum_{i=0}^2 x_iy_i=0$.
Substitute in $f$ and $g$ the expression
$-\frac{y_1}{y_0}x_1-\frac{y_2}{y_0}x_2$ at the place of $x_0$,
after getting rid of the denominators we may define
$\tilde f(x_1,x_2)=f(-{y_1}x_1-{y_2}x_2,y_0x_1,y_0x_2)$ and 
$\tilde g(x_1,x_2)=g(-{y_1}x_1-{y_2}x_2,y_0x_1,y_0x_2)$.

Then,  the expression
\begin{equation}\label{conuzero}
{\tilde f}_{04}{\tilde g}_{40}-4{\tilde f}_{13}{\tilde g}_{31}+6{\tilde f}_{22}{\tilde g}_{22}
-4{\tilde f}_{31}{\tilde g}_{13}+{\tilde f}_{40}{\tilde g}_{04}
\end{equation}

is equivalent to $y_0^4A(f,g,l^4)$. Differentiating by the coefficients
of $l^4$ and $f$ one finds easily $L_g$.
It is interesting that the whole expression (\ref{conuzero}) is divisible by $y_0^4$, while its individual summands are not.
\end{remark}
\begin{remark}\label{l1l2} Let $l_1$, $l_2$ be two lines.
$A\left(l_1^4,l_2^4,f\right)=0$
gives the condition that $f$ passes through the intersection point
$l_1=l_2=0$.
\end{remark}

Note also from Prop. \ref{alternativedef} that $A(f,f,l^4)=0$ if and only if $f$ cuts $l$ in an
equianharmonic $4$-tuple.
The quartic curve $A(f,f,*)$ in the dual space  is called the {\it equianharmonic envelope} of $f$, and consists of all lines which are cut by $f$
in a equianharmonic $4$-tuple .

This gives the classical geometric interpretation of the cubic invariant for plane quartics.
The condition $A(f,f,f)=0$ means that $f$ is apolar with its own
equianharmonic envelope (see \cite{Cia}), note that it gives a solution to Exercise (1) in the last page
of \cite{Stu}.

\section{Clebsch and L\"uroth quartics}
A plane quartic $f\in S^4V$ is called {\it Clebsch} if it has an apolar conic,
that is if there exists a nonzero $q\in S^2V^{\vee}$ such that $q\cdot f=0$.

One defines, for any $f\in S^4V$, the catalecticant map
$C_f\colon S^2V^{\vee}\to S^2V$
which is the contraction by $f$.

The equation of the Clebsch invariant
is easily to be seen as the determinant of $C_f$, that is we have(\cite{DK}, example (2.7))

\begin{theorem}[Clebsch]
A plane quartic $f$ is Clebsch if and only if $\det C_f=0$.
The conics which are apolar to $f$ are the elements of $\ker~C_f$.
\end{theorem}

It follows (\cite{D}, Lemma 6.3.22) that the general Clebsch quartic  can be expressed
as a sum of five $4$-th powers, that is
\begin{equation}\label{clebschexpression}
f=\sum_{i=0}^4l_i^4
\end{equation}
A general Clebsch quartic $f$ can be expressed 
as a sum of five $4$-th powers in $\infty^1$
many ways. Precisely the $5$ lines $l_i$ belong to a unique
smooth conic $Q$ in the dual plane, which is apolar to $f$ and it is found  as the generator
of  $\ker~C_f$. Equivalently, the $5$ lines $l_i$ are tangent to a unique conic, which is the dual conic of $Q$.

We recall that a {\it theta characteristic} on a general plane quartic $f$ is
a line bundle $\theta$ on $f$ such that $\theta^2$ is the canonical bundle.
Hence $\deg\theta=2$. There are $64$ theta characteristic on $f$.
If the curve is general, every bitangent is tangent in two distinct points $P_1$ and $P_2$, and the divisor $P_1+P_2$
defines a theta characteristic $\theta$ such that $h^0(\theta)=1$,
these are called odd theta characteristic and there are $28$ of them.
The remaining $36$ theta characteristic $\theta$ are called even
and they satisfy $h^0(\theta)=0$.

The {\it Scorza map} is the rational map from $\P^{14}=\P(S^4V)$ to itself which associates
to  a quartic $f$ the quartic $S(f)=\{x\in\P(V)|Ar(P_x(f))=0\}$,
where $P_x(f)$ is the cubic polar to $f$ at $x$ and $Ar$ is the Aronhold invariant \cite{Sc, DK, D}.
A convenient way to write explicitly the Scorza map
is through the expression of the Aronhold invariant  in \cite{LO}, example 1.2.1.
It is well known that the Scorza map  is a $36:1$ map.
Indeed the curve $S(f)$ is equipped with an even theta characteristic.
For a general quartic curve, its $36$ inverse images through the Scorza map
give all the $36$ even theta characteristic on $S(f)$.

A {\it L\"uroth quartic} is a 
 plane quartic containing the ten vertices of a complete 
{\it pentalateral}, or the limit of such curves.

If $l_i$ for $i=0,\ldots, 4$ are the lines of the pentalateral,
we may consider them as divisors (of degree $4$) over the curve.
Then $l_0+\ldots +l_4$ consists of $10$ double points,
the meeting points of the $5$ lines.
Let $P_1+\ldots +P_{10}$ the corresponding reduced divisor of degree $10$.
Then $P_1+\ldots +P_{10}=2H+\theta$
where $H$ is the hyperplane divisor and $\theta$ is a even theta characteristic,
which is called the pentalateral theta.
The pentalateral theta was called pentagonal theta in \cite{DK},
and it coincides with \cite{D}, Definition 6.3.30 (see the comments thereafter).

The following result is classical\cite{Sc},  for a modern proof see \cite{DK, D}.

\begin{prop}\label{clebschluroth}
Let $f$ be a Clebsch quartic with apolar conic $Q$,
 then $S(f)$ is a L\"uroth quartic equipped with the pentalateral theta corresponding to $Q$.
\end{prop}

The number of pentalateral theta on a general
L\"uroth quartic, called $\delta$ in the introduction, is equal to the degree
of the Scorza map when restricted to the hypersurface of Clebsch quartics.

Explicitly, if $f$ is Clebsch with equation
$$l_0^4+\ldots+l_4^4=0$$
then $S(f)$ has equation
$$\sum_{i=0}^4k_i\prod_{j\neq i}l_j=0$$
where $k_i=\prod_{p<q<r,i\notin\{p,q,r\}}|l_pl_ql_r|$
(see \cite{D}, Lemma 6.3.26)
so that $l_0,\ldots , l_4$ is a pentalateral inscribed in $S(f)$.
Note that the conic where the five lines which are the summands of $f$ are tangent,
is the same conic where the pentalateral inscribed in $S(f)$ is tangent.

The starting point of White-Miller paper \cite{WM} is the following 
remarkable characterization.

\begin{prop}\label{apolar4}
Let $f$ be a general Clebsch quartic and let $Q$ be the conic apolar to $f$. 
Let $L$ be a line.
 $L$ belongs to  $Q$ $\Longleftrightarrow$
the two binary quartic forms
$f_{|L}$ and $S(f)_{|L}$ are apolar.
\end{prop}
\proof $\Longrightarrow$ 
Let $L=l_4$ in the expression (\ref{clebschexpression}).
Then $f_{|L}=\sum_{i=0}^3l_i^4$
and $S(f)_{|L}=\prod_{i=0}^3{l_i}$. So the result is immediate
by the Lemma \ref{divides}.

$\Longleftarrow$ The proof is an explicit computation.
Let $f=\sum_{i=0}^4(\alpha_{0i}x_0+\alpha_{1i}x_1+\alpha_{2i}x_2)^4$.
Let $x_0=0$ be the equation of $L$.
The condition that $L$ belongs to $Q$ can be expressed as the vanishing of the degree $10$ polynomial
in the $\alpha_{ij}$ given by the determinant
$$D(\ldots,\alpha_{ij},\ldots)=\left|\begin{array}{cccccc}1&0&0&0&0&0\\
\alpha_{00}^2&\alpha_{00}\alpha_{10}&&&&\alpha_{20}^2\\
\alpha_{01}^2&\alpha_{01}\alpha_{11}&&&&\alpha_{21}^2\\
\vdots&\vdots&&&&\vdots\\
\alpha_{04}^2&\alpha_{04}\alpha_{14}&&&&\alpha_{24}^2\\\end{array}\right|$$
Let $g=S(f)$, $\tilde f=f(0,x_1,x_2)$, $\tilde g=g(0,x_1,x_2)$.
The condition that $f_{|L}$ and $S(f)_{|L}$ are apolar can be expressed as the vanishing of
the degree 20 polynomial
$$P(\ldots,\alpha_{ij},\ldots)={\tilde f}_{04}{\tilde g}_{40}-4{\tilde f}_{13}{\tilde g}_{31}+6{\tilde f}_{22}{\tilde g}_{22}
-4{\tilde f}_{31}{\tilde g}_{13}+{\tilde f}_{40}{\tilde g}_{04}$$
By the previous implication, we already know that $D$ divides $P$.
An explicit computation with Macaulay2 shows that
(up to scalar multiples)  $D(\ldots,\alpha_{ij},\ldots)^2=P(\ldots,\alpha_{ij},\ldots)$.
This proves both the implications at once.
\qed

\section{The White-Miller quartic $\mathrm{WM}_f$}
\begin{prop}\label{recoverf}Assume $f, g\in S^4V$ and $\det L_g\neq 0$
(see (\ref{lg}). Let's define $H\in S^4V^{\vee}$
by $H=A(f,g,*)$.
Then $f$ can be recovered by $g$ and $H$ as the expression of the following bordered determinant
$$f(x)=\left|\begin{array}{cc}L_g&{\overline{H}}\\
(x)^4&0\\
\end{array}\right|$$
where $(x)^4$ denotes the $4$-th symmetric power of the row vector $(x_0, x_1, x_2)$
containing all terms of the form $\frac{4!}{i!j!k!}x_0^ix_1^jx_2^k$ for $i+j+k=4$.
\end{prop}
\begin{proof}
The hypothesis reads as $L_{g}\overline{f}=\overline{H}$.
The assertion reads as $\overline{f}=L_{g}^{-1}\overline{H}$.
\end{proof}
\vskip 0.5cm

\begin{prop}[White--Miller]
Let $f\in S^4V$ be a Clebsch quartic with apolar conic $Q\in S^2V^{\vee}$.
Let $g=S(f)$ be the associated L\"uroth quartic (see Prop. \ref{clebschluroth}).
Then, up to scalar multiples,  $A(f,S(f),*)=Q^2\in S^4V^{\vee}$. 
\end{prop}
\begin{proof}
It is a reformulation of Prop. \ref{apolar4} and Prop. \ref{alternativedef}.
\end{proof}
\vskip 0.5cm

\begin{corollary}\label{recoverclebsch}
Let $f$ be a Clebsch quartic with apolar conic $Q$ and $g=S(f)$ be L\"uroth.

(i) If $\det L_g\neq 0$, then the expression of $f$ can be recovered, up to scalar multiples, from $g$ and $Q$ by the formula
$$f=\left|\begin{array}{cc}L_g&\overline{Q^2}\\
(x)^4&0\\
\end{array}\right|$$

(ii) If $\det L_f\neq 0$, then the expression of $g$ can be recovered, up to scalar multiples, from $f$ and $Q$ by the formula
$$g=\left|\begin{array}{cc}L_f&\overline{Q^2}\\
(x)^4&0\\
\end{array}\right|$$
\end{corollary}

\begin{corollary}\label{glurothidentity}
Let  $g$ be a L\"uroth quartic with pentalateral theta corresponding to the conic $Q$.
Then the following identity holds for every 
$(x)^2=(x_0^2,\ldots, x_2^2)$
$$\left|\begin{array}{cc}L_g&\overline {Q^2}\\
(x)^2{\overline{Q}}^t&0\\
\end{array}\right|=0$$
\end{corollary}
\begin{proof}
Write that $f\cdot Q=0$ in Corollary \ref{recoverclebsch}.
\end{proof}
\vskip 0.5cm

The previous corollary gives the motivation for the following definition.
\begin{definition} For any $f\in S^4V$,
the White-Miller quartic $\mathrm{WM}_f$
is defined in $\P(S^2V)$ with coordinates $(q_0,\ldots q_5)$ by the formula
$$\mathrm{WM}_f=\left|\begin{array}{cc}L_f&(q^2)^t\\
q^2&0\\
\end{array}\right|$$
where $(q^2)$ is the vector of the $15$ cooefficients
of a double conic. Note that $\mathrm{WM}_f$  corresponds
to the unique irreducible summand isomorphic to $S^4V^{\vee}$ inside $S^4(S^2V)$.
\end{definition}

The definition is taken verbatim from \cite{WM},
it is equivalent to
$$\mathrm{WM}_f=q^2\left(\textrm{ad}L_f\right)(q^2)^t$$
where $\textrm{ad}L_f$ is the adjoint matrix of $L_f$.

\begin{prop}[White--Miller]\label{pentasingular}
Let  $f$ be a L\"uroth quartic with pentalateral theta corresponding to the conic $Q$.
Then $Q$ is a singular point of $\mathrm{WM}_f$.
\end{prop}
\begin{proof}
The six partial derivatives of
$$\left|\begin{array}{cc}L_f&q^2\\
q^2&0\\
\end{array}\right|$$
computed in the point $Q$ vanish if and only if 
$$\left|\begin{array}{cc}L_g&\overline{Q^2}\\
(x)^2{\overline{Q}}^t&0\\
\end{array}\right|=0$$
for every $(x)^2$. This identity holds by Corollary \ref{glurothidentity}.
\end{proof}
\vskip 0.5cm

If $\mathrm{WM}_f$ is not singular elsewhere,
then the discriminant of $\mathrm{WM}_f$ should detect
if $f$ is L\"uroth. This was the aim of White-Miller construction.
In the next section we see that there are at least other $28$ singular points
for $\mathrm{WM}_f$, so that its discriminant vanish identically.

\section{The proof of Theorem \ref{28bit} and the algorithm to detect if a quartic is L\"uroth}
\label{algorithm}
\begin{prop}\label{abitangent}
Let $f, l^4\in S^4V$ and consider  $A(f,l^4,*)\in S^4V^{\vee}$ as a quartic in $\P(V^{\vee})$.

(i)  $A(f,l^4,*)$ 
splits in four lines concurrent lines, corresponding to the four intersection points
where $f$ and $l$ meet.

(ii) $A(f,l^4,*)$ is a double conic if and only if $l$ is a bitangent to $f$.
\end{prop}
\begin{proof}
(i) follows from Prop. \ref{alternativedef} (ii) and Remark \ref{l1l2}.
(ii) is immediate from (i).
\end{proof}
\vskip 0.5cm

\begin{definition}\label{qlf}
Let $l$ be a bitangent to $f$. It defines a reducible conic $Q_{l,f}$
in $\P(V^{\vee})$ given by the two pencils through the two points of tangency.
On the algebraic side we have, from Prop. \ref{abitangent} (ii), the equivalent definition
from the identity (up to scalar multiples)
$$Q_{l,f}^2=A(f,l^4,*)$$

\end{definition}

The following Proposition proves the Theorem \ref{28bit}.

\begin{prop}
Let $f\in S^4V$ and $\det L_f\neq 0$. Let $l$ be a bitangent to $f$
and let $Q_{l,f}$ be the corresponding reducible conic of Def. \ref{qlf}.

The quartic $\mathrm{WM}_f$ is singular in $Q_{l,f}$. 
\end{prop}
\begin{proof}
From Prop. \ref{recoverf}
it follows that $l^4$ is the expression of the bordered determinant
$$\left|\begin{array}{cc}L_f&\overline{Q_{l,f}^2}\\
(x)^4&0\\
\end{array}\right|$$
 
Since $Q_{l,f}$ splits in two pencils of lines through points of $l$, we get the relation
$Q_{l,f}\cdot l^4=0$,
it follows
$$\left|\begin{array}{cc}L_f&\overline{Q_{l,f}^2}\\
(x)^2\overline{Q_{l,f}}&0\\
\end{array}\right|=0\quad\forall (x)^2$$
that implies the thesis.
\end{proof}
\vskip 0.5cm

The algorithm sketched in the introduction,
based on the analysis of the singular locus of $\mathrm{WM}_f$, works as follows.

$\bullet$(i) if the singular locus of $\mathrm{WM}_f$ has dimension zero and degree $28$, then $f$ {\it is not L\"uroth}
because otherwise the pentalateral theta should give a $29$th singular point by Prop. \ref{pentasingular}.

$\bullet$(ii) if the singular locus of $\mathrm{WM}_f$ contains a point corresponding to a smooth conic $Q$,
then by Prop. \ref{recoverf} the formula $$g=\left|\begin{array}{cc}L_f&Q^2\\
(x)^4&0\\
\end{array}\right|$$ defines a quartic which is apolar to $Q$, hence it is Clebsch.

If $\det L_g\neq 0$, from Corollary \ref{recoverclebsch} it follows that 
$$f=\left|\begin{array}{cc}L_g&Q^2\\
(x)^4&0\\
\end{array}\right|$$ and in particular $f$ coincides with $S(g)$ and it is L\"uroth.

Moreover, assume that the rank of the catalecticant $C_g$ is $5$.
Fixing a line $l_0$ which belongs to $Q$, we get that $g'=g-\alpha_0 l_0^4$ has rank $4$ for a convenient $\alpha_0\in\C$.
The kernel of the catalecticant $C_{g'}$ is generated by two conics, which meet in $4$ points,
which correspond to $l_1,\ldots, l_4$.
By a classical technique,
which goes back to Sylvester, there exist scalars $\alpha_i$ for $i=1,\ldots, 4$ such that $g=\sum_{i=0}^4\alpha_il_i^4$(\cite{RS})
and $l_i$ for $i=0,\ldots, 4$ give a pentalateral inscribed to $f$.

$\bullet$(iii) In the case where the singular locus of $\mathrm{WM}_f$
is bigger than $29$ points, or in the case (ii) when $\det L_g= 0$,
a further analysis is necessary and only partial results can be achieved by our algorithm.
Some examples in this third category, either L\"uroth or not L\"uroth,
have been analyzed at the end of the introduction.
Note that if $f$ is a desmic quartic then $\det L_f=0$ and Corollary \ref{recoverclebsch} does not apply.
More examples are in the Section \ref{computationalfacts}.

\section{Computational facts and examples}\label{computationalfacts}

\begin{theorem}[Morley]\label{morley}
The L\"uroth invariant has degree $54$.
\end{theorem}
\begin{proof}

The Scorza map is a rational map from $\P^{14}$ to $\P^{14}$
with entries of degree $4$ in the $15$ homogeneous coefficients of a quartic $f$.
We  call $L$ the degree of the L\"uroth invariant.
From the discussion after Prop. \ref{clebschluroth}, 
since the Clebsch invariant has degree $6$ we get the well known equality

$$\frac{4L}{36/\delta}=6$$

hence $54=\delta L$.
This equality coincides with the equation (2.5) in \cite{Tyu} specialized to $r=5$, $d=4$, where 
$s_4(5)=54$(see \cite{Tyu}\S 5), $\deg p_C$ is our $\delta$
 and $p_C(\mathrm{MP}^4_5)$ is the L\"uroth hypersurface in $\P^{14}$.
 
If $$f=x_0x_1x_2(x_0+x_1+x_2)+(x_0+2x_1+3x_2)\left(x_0x_1x_2+(x_0x_1+x_0x_2+x_1x_2)(x_0+x_1+x_2)\right)$$
then a Macaulay2 computation says that$\textrm{WM}_f$ has exactly $29$ singular points.
This implies that the L\"uroth quartic $f$ has a unique pentalateral theta,
corresponding to the conic tangent to the pentalateral
given by $l_0=x_0$, $l_1=x_1$, $l_2=x_2$, $l_3=x_0+x_1+x_2$, $l_4=x_0+2x_1+3x_2$.
By semicontinuity, the general L\"uroth quartic curve has a unique pentalateral theta. 
This implies $\delta=1$ and the theorem follows. 
\end{proof}
\vskip 0.5cm

In the example
$f=x_2^2(x_0^2+x_1^2)+x_2(x_0^3+x_1^3)-x_0^3x_1+(1/2)x_0^2x_1^2-x_0x_1^3$,
found in \cite{PT} section 9.3 (put $c=-\frac{1}{4}$),
the singular locus of $\mathrm{WM}_f$ consists of
a conic, whose general point corresponds to a singular conic, plus $8$ points,
 each one corresponding to a smooth conic. This example has
$8$ distinct pentalateral theta.

In the example
$f=25(x_0^4+x_1^4+x_2^4)-34(x_0^2x_1^2+x_0^2x_2^2+x_1^2x_2^2)$,
considered in \cite{PSV}, and previously by Edge,
the singular locus of $\mathrm{WM}_f$ consists of
$40$ points, consisting of $28$ points giving the bitangents plus $12$ points,
 each one corresponding to a smooth conic. This example has
$12$ distinct pentalateral theta.

The singular L\"uroth quartics form two irreducible components $L_1$ and $L_2$ (\cite{OS2}).
The White-Miller quartic of a singular L\"uroth quartic in $L_2$ (with the pentalateral having three concurrent lines)
is singular along a conic in $\P^5$ corresponding to reducible conics and at seven points,
each one corresponding to a smooth conic. 
%So they can be found through the Scorza map
%from seven different Clebsch quartics .

According to \cite{hB14}, pag. 368, desmic quartics have equation
$f(x_0,x_1,x_2)=\left((a^2-2mbc)x_0+(b^2-2mca)x_1+(c^2-2mab)x_2\right)x_0x_1x_2-m^2(a^2x_0^4+b^2x_1^4+c^2x_2^4-2bcx_1^2x_2^2-2cax_0^2x_2^2-2abx_0^2x_1^2)$
with parameters $a, b, c, m$.
An explicit computation shows that $\textrm{rk\ } L_f=14$,
hence $\textrm{rk\ }\textrm{ad }L_f=1$ and $\mathrm{WM}_f$ is a double (hyper)quadric.

Klein quartic $f(x_0,x_1,x_2)=x_0^3x_1+x_1^3x_2+x_2^3x_0$ is not L\"uroth,
the singular locus of $\mathrm{WM}_f$ is just given by $28$ points, as in the general case.

For the same reason, the following curve, with a cusp in $(0,0,1)$,
is not L\"uroth
$f(x_0,x_1,x_2)=x_0^4+2x_1^4-34x_0^2x_1^2+(x_0^2x_2^2+x_1^2x_2^2+2x_0x_1x_2^2)+41x_0^3x_1+51x_1^3x_2+21x_0x_1^3-11x_1x_2^3$.

If $f$ is a  Caporali quartic (sum of four $4$-th powers) we have that 
$\textrm{rk\ } L_f=14$ and
 $\mathrm{WM}_f$ is a double (hyper)quadric.

\section{Appendix: the matrix $L_g$ of formula (\ref{lg})}\label{appendix}

To print the matrix $L_g$ is convenient to use the notation

{\tiny
$$g=x^{4} {g}_{0}+x^{3} y {g}_{1}+x^{3} z {g}_{2}+x^{2} y^{2} {g}_{3}+x^{2}
      y z {g}_{4}+x^{2} z^{2} {g}_{5}+x y^{3} {g}_{6}+x y^{2} z {g}_{7}+x y
      z^{2} {g}_{8}+x z^{3} {g}_{9}+y^{4} {g}_{10}+y^{3} z {g}_{11}+y^{2} z^{2}
      {g}_{12}+y z^{3} {g}_{13}+z^{4} {g}_{14}$$}

and we get

{\tiny
$$
L_g=\bgroup\makeatletter\c@MaxMatrixCols=15\makeatother\begin{pmatrix}0&
      0&
      0&
      0&
      0&
      0&
      0&
      0&
      0&
      0&
      144 {g}_{14}&
      {-36 {g}_{13}}&
      24 {g}_{12}&
      {-36 {g}_{11}}&
      144 {g}_{10}\\
      0&
      0&
      0&
      0&
      0&
      0&
      {-36 {g}_{14}}&
      9 {g}_{13}&
      {-6 {g}_{12}}&
      9 {g}_{11}&
      0&
      9 {g}_{9}&
      {-6 {g}_{8}}&
      9 {g}_{7}&
      {-36 {g}_{6}}\\
      0&
      0&
      0&
      0&
      0&
      0&
      9 {g}_{13}&
      {-6 {g}_{12}}&
      9 {g}_{11}&
      {-36 {g}_{10}}&
      {-36 {g}_{9}}&
      9 {g}_{8}&
      {-6 {g}_{7}}&
      9 {g}_{6}&
      0\\
      0&
      0&
      0&
      24 {g}_{14}&
      {-6 {g}_{13}}&
      4 {g}_{12}&
      0&
      {-6 {g}_{9}}&
      4 {g}_{8}&
      {-6 {g}_{7}}&
      0&
      0&
      4 {g}_{5}&
      {-6 {g}_{4}}&
      24 {g}_{3}\\
      0&
      0&
      0&
      {-6 {g}_{13}}&
      4 {g}_{12}&
      {-6 {g}_{11}}&
      9 {g}_{9}&
      {-{g}_{8}}&
      {-{g}_{7}}&
      9 {g}_{6}&
      0&
      {-6 {g}_{5}}&
      4 {g}_{4}&
      {-6 {g}_{3}}&
      0\\
      0&
      0&
      0&
      4 {g}_{12}&
      {-6 {g}_{11}}&
      24 {g}_{10}&
      {-6 {g}_{8}}&
      4 {g}_{7}&
      {-6 {g}_{6}}&
      0&
      24 {g}_{5}&
      {-6 {g}_{4}}&
      4 {g}_{3}&
      0&
      0\\
      0&
      {-36 {g}_{14}}&
      9 {g}_{13}&
      0&
      9 {g}_{9}&
      {-6 {g}_{8}}&
      0&
      0&
      {-6 {g}_{5}}&
      9 {g}_{4}&
      0&
      0&
      0&
      9 {g}_{2}&
      {-36 {g}_{1}}\\
      0&
      9 {g}_{13}&
      {-6 {g}_{12}}&
      {-6 {g}_{9}}&
      {-{g}_{8}}&
      4 {g}_{7}&
      0&
      4 {g}_{5}&
      {-{g}_{4}}&
      {-6 {g}_{3}}&
      0&
      0&
      {-6 {g}_{2}}&
      9 {g}_{1}&
      0\\
      0&
      {-6 {g}_{12}}&
      9 {g}_{11}&
      4 {g}_{8}&
      {-{g}_{7}}&
      {-6 {g}_{6}}&
      {-6 {g}_{5}}&
      {-{g}_{4}}&
      4 {g}_{3}&
      0&
      0&
      9 {g}_{2}&
      {-6 {g}_{1}}&
      0&
      0\\
      0&
      9 {g}_{11}&
      {-36 {g}_{10}}&
      {-6 {g}_{7}}&
      9 {g}_{6}&
      0&
      9 {g}_{4}&
      {-6 {g}_{3}}&
      0&
      0&
      {-36 {g}_{2}}&
      9 {g}_{1}&
      0&
      0&
      0\\
      144 {g}_{14}&
      0&
      {-36 {g}_{9}}&
      0&
      0&
      24 {g}_{5}&
      0&
      0&
      0&
      {-36 {g}_{2}}&
      0&
      0&
      0&
      0&
      144 {g}_{0}\\
      {-36 {g}_{13}}&
      9 {g}_{9}&
      9 {g}_{8}&
      0&
      {-6 {g}_{5}}&
      {-6 {g}_{4}}&
      0&
      0&
      9 {g}_{2}&
      9 {g}_{1}&
      0&
      0&
      0&
      {-36 {g}_{0}}&
      0\\
      24 {g}_{12}&
      {-6 {g}_{8}}&
      {-6 {g}_{7}}&
      4 {g}_{5}&
      4 {g}_{4}&
      4 {g}_{3}&
      0&
      {-6 {g}_{2}}&
      {-6 {g}_{1}}&
      0&
      0&
      0&
      24 {g}_{0}&
      0&
      0\\
      {-36 {g}_{11}}&
      9 {g}_{7}&
      9 {g}_{6}&
      {-6 {g}_{4}}&
      {-6 {g}_{3}}&
      0&
      9 {g}_{2}&
      9 {g}_{1}&
      0&
      0&
      0&
      {-36 {g}_{0}}&
      0&
      0&
      0\\
      144 {g}_{10}&
      {-36 {g}_{6}}&
      0&
      24 {g}_{3}&
      0&
      0&
      {-36 {g}_{1}}&
      0&
      0&
      0&
      144 {g}_{0}&
      0&
      0&
      0&
      0\\
      \end{pmatrix}\egroup$$
}

 \noindent
  \textsc{g. ottaviani} -
  Dipartimento di Matematica ``U. Dini'', Universit\`a di Firenze, viale Morgagni 67/A, 50134 Firenze (Italy). e-mail: \texttt{ottavian@math.unifi.it}
  
  \separation
  \noindent

\end{document}